\documentclass[12pt,reqno]{article}
\usepackage{amssymb}
\usepackage{anysize}

\usepackage[usenames]{color}
\usepackage{amssymb}
\usepackage{graphicx}
\usepackage{amscd}

\usepackage[colorlinks=true,
linkcolor=webgreen,
filecolor=webbrown,
citecolor=webgreen]{hyperref}

\definecolor{webgreen}{rgb}{0,.5,0}
\definecolor{webbrown}{rgb}{.6,0,0}

\usepackage{color}
\usepackage{fullpage}
\usepackage{float}

\usepackage{graphics,amsmath,amssymb}
\usepackage{amsthm}
\usepackage{amsfonts}
\usepackage{latexsym}
\usepackage{epsf}

\usepackage{mathrsfs}

\newcommand{\Z}{\mathbb{Z}}
\newcommand{\Q}{\mathbb{Q}}

\newcommand{\sign}{\mathrm{sign}}

\newcommand{\seqnum}[1]{\href{http://oeis.org/#1}{\underline{#1}}}

\begin{document}

\theoremstyle{plain}
\newtheorem{theorem}{Theorem}
\newtheorem{corollary}{Corollary}
\newtheorem{lemma}{Lemma}
\newtheorem{proposition}{Proposition}

\theoremstyle{definition}
\newtheorem{definition}{Definition}
\newtheorem{example}{Example}
\newtheorem{conjecture}{Conjecture}
\newtheorem{problem}{Problem}

\theoremstyle{remark}
\newtheorem{remark}{Remark}

\begin{center}
\vskip 1cm{\LARGE\bf
On arithmetic progressions in\\
\vskip .12in
Lucas sequences}
\vskip 1cm
\large
Lajos Hajdu \footnote{The author was supported by the NKFIH grant 115479.}\\
Institute of Mathematics,\\
University of Debrecen, \\
P.O. Box 400,\\
H-4002 Debrecen,\\
Hungary\\
\href{mailto:hajdul@science.unideb.hu}{\tt hajdul@science.unideb.hu}\\

\ \\
M\'arton Szikszai\footnote{The author was supported by the Hungarian Academy of Sciences.}\\
Institute of Mathematics, University of Debrecen\\
and\\
MTA-DE Research Group Equations Functions and Curves, Hungarian Academy of Sciences and University of Debrecen\\
P.O. Box 400,\\
H-4002 Debrecen,\\
Hungary\\
\href{mailto:hajdul@science.unideb.hu}{\tt hajdul@science.unideb.hu}\\

\ \\
Volker Ziegler\footnote{The author was supported by the Austrian Research Foundation (FWF), Project P24801-N26.}\\
Institute of Mathematics,\\
University of Salzburg,\\
Hellbrunnerstrasse 34/I,\\
A-5020 Salzburg,\\
Austria\\
\href{mailto:volker.ziegler@sbg.ac.at}{\tt volker.ziegler@sbg.ac.at}\\
\end{center}

\begin{abstract}
In this paper, we consider arithmetic progressions contained in Lucas sequences
of the first and second kind. We prove that for almost all Lucas sequences, there are only
finitely many arithmetic three term progressions and their number can be effectively bounded.
We also show that there are only a few Lucas sequences which contain infinitely many arithmetic three term progressions
and one can explicitly list both the sequences and the progressions in them. A more precise
statement is given for sequences with dominant root.
\end{abstract}

\section{Introduction}
The study of additive structures in certain sets of integers has a long history.
In particular, the description of arithmetic progressions has attracted the attention of many researchers
and is still an actively studied topic with a
vast literature. Without trying to be exhaustive and going into details we only
mention a few interesting directions.

Bremner \cite{bremner} found infinitely many elliptic curves such that among the
first coordinates of rational points one can find an eight-term arithmetic
progression. Campbell \cite{campbell}, modifying ideas of Bremner, gave the
first example of a progression of length twelve which was later improved to
fourteen by MacLeod \cite{macleod}. For more details on the various
generalizations to other families of curves and recent progress we refer to the
paper of Ciss and Moody \cite{cissmoody} and the references given there.

The study of arithmetic progressions in solution sets of Pellian equations is another interesting direction.
Dujella, Peth\H{o} and Tadi\'c \cite{dujellapethotadic} showed that for every four term arithmetic progression $(y_1,y_2,y_3,y_4)\neq (0,1,2,3)$ there exist infinitely many pairs $(d,m)$ such that the equation $x^2-dy^2=m$ admits solutions with $y=y_i$ for $i=1,2,3,4$.
On the other hand, Peth\H{o} and Ziegler \cite{pethoziegler} proved that for five-term
arithmetic progressions only a finite number of such equations are possible. For
further progress we refer to the papers of Aguirre, Dujella and Peral
\cite{aguirredujellaperal} and Gonz\'alez-Jim\'enez \cite{gonzalezjimenez}.

In this paper, we connect to the field by considering arithmetic progressions in
Lucas sequences. Let $A$ and $B$ be non-zero integers such that the roots
$\alpha$ and $\beta$ of the polynomial
\begin{equation}
\label{charpoly}
x^2-Ax-B
\end{equation}
satisfy that $\alpha/\beta$ is not a root of unity. Define the sequences
$u=(u_n)_{n=0}^\infty$ and $v=(v_n)_{n=0}^\infty$ via the binary recurrence
relations
\begin{equation}
\label{recurrence}
u_{n+2}=Au_{n+1}+Bu_n\qquad \mathrm{and} \qquad v_{n+2}=Av_{n+1}+Bv_n \qquad (n\geq 0)
\end{equation}
with initial values $u_0=0,\ u_1=1$ and $v_0=2,\ v_1=A$. We call $u$ and $v$ the
Lucas sequences of first and second kind corresponding to the pair $(A,B)$,
respectively.

\begin{remark}
The classical definition of $u$ and $v$ requires $\gcd(A,B)=1$ which is needed for
certain $p$-adic properties to hold. Hence our definition is more general.
\end{remark}

\begin{remark}
The assumption on the roots of the companion polynomial \eqref{charpoly}
coincides with the non-degeneracy property. The study of all linear recurrences
can effectively be reduced to such sequences, see
\cite{everestpoortenshparlinskiward}.
\end{remark}

The investigation of arithmetic progressions in linear recurrences is not a completely new topic.
Pint\'er and Ziegler \cite{pinterziegler} gave a criterion
for linear recurrences of general order to contain infinitely many three term
arithmetic progressions. Note that the study of three term arithmetic
progressions is the first non-trivial problem and is the most general.
Further, they proved that the
sequence of Fibonacci and Jacobsthal numbers are the only increasing Lucas
sequences having infinitely many three term progressions.

In this paper, we connect to the results of Pint\'er and Ziegler. On one hand,
we give an upper bound for the number of three term arithmetic progressions provided
that there are only finitely many. On the other hand, we explicitly list all those sequences which
contain infinitely many arithmetic progressions. Finally, by restricting
ourselves to sequences with companion polynomials having a dominant root, we
find all sequences which contain a three term arithmetic progression and give a
complete list of the occurrences.

Our main result is as follows:

\begin{theorem}
\label{theorem1}
Let $u=(u_n)_{n=0}^\infty$ and $v=(v_n)_{n=0}^\infty$ be the Lucas sequences of
first and second kind, respectively, corresponding to the pair $(A,B)$. Then $u$
and $v$ admit at most $6.45\cdot 10^{2340}$ triples $(k,l,m)$ such that $u_k<u_l<u_m$ is
a three term arithmetic progression,
except when $u$ corresponds to $(A,B)=(\pm 1,1),(\pm 1,2),(-1,-2)$ and $v$
corresponds to $(A,B)=(\pm1, 1),(-1,\pm 2)$, in which cases they admit
infinitely many.
\end{theorem}

\begin{remark}
We expect that the upper bound in Theorem \ref{theorem1} is very far from being sharp.
It would be an interesting problem to study the average number of three term arithmetic progressions in Lucas sequences of the first and second kind.
\end{remark}

By imposing stronger restrictions on the sequences we can state
much more. Namely, if the characteristic polynomial \eqref{charpoly} has a
dominant root, then it is possible to explicitly list all three term arithmetic
progressions in the corresponding sequence.

\begin{theorem}
\label{theorem2}
Let $u=(u_n)_{n=0}^\infty$ and $v=(v_n)_{n=0}^\infty$ be the Lucas sequences of
first and second kind, respectively, corresponding to the pair $(A,B)$.
Assume that $A^2+4B>0$. Then $u$ and $v$ admit no three term
arithmetic progression, except the cases listed in Tables \ref{table1} and
\ref{table2}, respectively.
\end{theorem}

\begin{table}[ht]
\centering
\begin{tabular}{|c|c|}
\hline
$(A,B)$ & $(k,l,m)$\\
\hline
$(1,1)$ & $(0,1,3),\ (2,3,4),\ (t,t+2,t+3),\ t\geq 0$\\
\hline
$(-1,1)$ & $(1,0,2),\ (t,t+1,t+3),\ t\geq 0$\\
\hline
$(1,2)$ & $(1,2t+1,2t+2),\ (2,2t+1,2t+2),\ t\geq 1$\\
\hline
$(-1,2)$ & $(t+2,t,t+1),\ t\geq 0$\\
\hline
$(2,B),\ B\geq 1$ & $(0,1,2)$\\
\hline
$(1,B),\ B\geq 3$ & $(1,3,4),\ (2,3,4)$\\
\hline
$(-1,B),\ B\geq 3$ & $(1,0,2)$\\
\hline
\end{tabular}
\caption{Triples $(k,l,m)$ s.t. $(u_k,u_l,u_m)$ is a three term AP.}
\label{table1}
\end{table}

\begin{table}[ht]
\centering
\begin{tabular}{|c|c|}
\hline
$(A,B)$ & $(k,l,m)$\\
\hline
$(1,1)$ & $(1,0,2),\ (t,t+2,t+3),\ t\geq 0$\\
\hline
$(-1,1)$ & $(t,t+1,t+3),\ t\geq 0$\\
\hline
$(-1,2)$ & $(t,t-1,t+1),\ t\geq 1$\\
\hline
$(-2,1)$ & $(1,0,2)$\\
\hline
$(1,3)$ & $(1,4,5)$\\
\hline
$(-3,-1)$ & $(1,0,2)$\\
\hline
\end{tabular}
\caption{Triples $(k,l,m)$ s.t. $(v_k,v_l,v_m)$ is a three term AP.}
\label{table2}
\end{table}

As an immediate consequence of Theorem \ref{theorem2} we have the following:

\begin{corollary}
\label{cor}
Let $u=(u_n)_{n=0}^\infty$ (resp. $v=(v_n)_{n=0}^\infty$) be the Lucas sequence
of first (resp. second) kind corresponding to the pair $(A,B)$. Assume that
$A^2+4B>0$. If $u$ (resp. $v$) contains more than two (resp. one)
three term arithmetic progressions, then $u$ (resp. $v$) contains infinitely
many three term arithmetic progressions.
\end{corollary}

The ineffective methods we use in the proof of Theorem \ref{theorem1} do not
give such sharp bounds as in Corollary \ref{cor}. However, it would be
interesting to find optimal bounds also in the general case, i.e., if we drop the
condition $A^2+4B>0$ in Corollary \ref{cor}.

In the next section, we prove Theorem \ref{theorem2}. The main idea is to show
that under a testable condition, the growth of the sequence would contradict the
existence of any arithmetic progression. Further, this condition leaves only
finitely many explicitly given possibilities for $(A,B)$ which we treat one by
one, usually in an
elementary way. The proof of Theorem \ref{theorem1} in Section \ref{Sec:imag} is
based on the theory of $S$-unit equations.

\section{The dominant root case}\label{Sec:real}

This section is devoted to the case, where $A^2+4B>0$, that is, the roots of the
companion polynomial are real. First, we remark that in any linear recurrence,
the terms can be represented as polynomial-exponential sums depending on the
roots of the companion polynomial. In our situation, this leads to the
well-known formulas
\[
u_n=\dfrac{\alpha^n-\beta^n}{\alpha-\beta} \qquad \mathrm{and} \qquad
v_n=\alpha^n+\beta^n \qquad (n\geq 0).
\]
We will regularly use this definition of Lucas sequences
instead of the recurrence relation \eqref{recurrence}.

The following lemma concerns equations satisfied by three term arithmetic
progressions.

\begin{lemma}
\label{lemmatech}
Let $u=(u_n)_{n=0}^\infty$ be any sequence. Suppose that for some $n_0$ we
have $|u_n/u_{n'}|>3$ for every $n>n_0$ and every $n'<n$. If $m>n_0$, then
neither of the
equations
\begin{equation}
\label{equation:klm}
u_k-2u_l+u_m=0, \qquad u_l-2u_k+u_m=0,\qquad u_k-2u_m+u_l=0
\end{equation}
has a solution $(k,l,m)$ with $k<l<m$. In particular, $u$ does not admit any
three term arithmetic progression with terms corresponding to the indices $k,l$
and $m$ with $\max\{k,l,m\}>n_0$.
\end{lemma}

\begin{proof}
Assuming that $n_0<m$ gives us either $3|u_k|<3|u_l|<|u_m|$ or
$3|u_l|<3|u_k|<|u_m|$. In the former case, we have
\[
|u_k-2u_l+u_m|\geq |u_m|-|u_k-2u_l|\geq |u_m|-(|u_k|+2|u_l|)>3|u_l|-3|u_l|=0.
\]
Thus the first equation of \eqref{equation:klm} has no solution with $k<l<m$.
In an analogous way, one can get the same deduction for $3|u_l|<3|u_k|<|u_m|$
and it is easy to see that the second and third equations of
\eqref{equation:klm} are unsolvable as well.
\end{proof}

We are in need of a good criterion for the applicability of Lemma
\ref{lemmatech}.

\begin{lemma}
\label{lem:growth_Lucas}
Let $u=(u_n)_{n=0}^\infty$ be the Lucas sequence of first kind corresponding to
the pair
$(A,B)$ with $A^2+4B>0$. Assume that $n>n'$. Then $|u_n/u_{n'}|>3$ if $n$ is odd
or $n\geq 8$, unless
\begin{itemize}
 \item $B<0$ and $|A|\leq 6$ or
 \item $|A|=1$ and $0<B\leq 9$ or
 \item $|A|=2$ and $0<B\leq 3$.
\end{itemize}
\end{lemma}

\begin{proof}
Under our conditions the polynomial $X^2-AX-B$ has a
dominant root, say $\alpha$. Note that if $|A|>6$ or $|B|>9$, then we have
$|\alpha|>3$ and may write
\[
\left|\dfrac{u_n}{u_{n-1}}\right|=\left|\dfrac{\dfrac{\alpha^n-\beta^n}{
\alpha-\beta}}{\dfrac{\alpha^{n-1}-\beta^{n-1}}{\alpha-\beta}}
\right|=\left|\dfrac{\alpha^n}{\alpha^{n-1}}\dfrac{1-\left(\frac{\beta}{\alpha}
\right)^{n}}{1-\left(\frac{\beta}{\alpha}\right)^{n-1}}
\right|=|\alpha|\left|\dfrac{1-\theta^n}{1-\theta^{n-1}}\right|>3\left|\dfrac{
1-\theta^n}{1-\theta^{n-1}}\right|
\]
with $\theta=\beta/\alpha$. It suffices to prove that
\[
\left|\dfrac{
1-\theta^n}{1-\theta^{n-1}}\right|>1.
\]
Observe that there are trivial cases. These include $\theta>0$ and also
$\theta<0$ when $n$ is odd. What remains to consider is the case when $\theta<0$
and $n$ is even.

Next, we show that the quotients $|u_{2n+2}/u_{2n+1}|$ increase with
$n$, that is, they satisfy the inequality
\[
\left|\dfrac{u_{2n+2}}{u_{2n+1}}\right|>\left|\dfrac{u_{2n}}{u_{2n-1}}\right|.
\]
Equivalently, we may write
\[
|\alpha|\left|\dfrac{1-\theta^{2n+2}}{1-\theta^{2n+1}}
\right|>|\alpha|\left|\dfrac{1-\theta^{2n}}{1-\theta^{2n-1}}\right|.
\]
Since $0<|\theta|<1$, it reduces to
\[
\dfrac{1-\theta^{2n+2}}{1-\theta^{2n+1}}>\dfrac{1-\theta^{2n}}{1-\theta^{2n-1}}
\]
giving us
\[
1-\theta^{2n-1}-\theta^{2n+2}+\theta^{4n+1}>1-\theta^{2n+1}-\theta^{2n}+\theta^{
4n+1}.
\]
Collecting the terms on the left hand side and dividing both sides of the
inequality by $-\theta^{2n-1}>0$ we obtain
\[
\theta^3-\theta^2-\theta+1=(\theta-1)^2(\theta+1)>0,
\]
which is obviously true for any $\theta$ with $0<|\theta|<1$. Since
$|u_2/u_1|=|A|\geq 1$, we have, in particular, that $|u_{n}/u_{n-1}|>1$ for all
$n\geq 0$. This proves the lemma for all odd $n$ and all $n>2n_0$ with
$|u_{2n_0}/u_{2n_0-1}|>3$. Thus we need to show that $|u_{8}/u_{7}|>3$ under our
hypotheses.

We distinguish between the cases where $B>0$ and $B<0$. Let us start with $B<0$.
We have
\[
|\theta|=\left|\frac{\beta}{\alpha}\right|=\frac{|A|-\sqrt{A^2+4B}}{|A|+\sqrt{
A^2+4B}}\leq \frac{|A|-1}{|A|+1}=1-\frac{2}{|A|+1}.
\]
Note that in this case we also have that $|\alpha|>\frac{|A|+1}2$. Write $x=\frac{|A|+1}2$.
We show that
\[
\left|\dfrac{u_{8}}{u_{7}}\right|\geq x \frac{1-\left(1-\dfrac
1x\right)^{8}}{1+\left(1-\dfrac 1x\right)^{7}}>3.
\]
Solving this inequality for $x$ shows that the second inequality holds unless
$x<3.55$, i.e., $|A|<6.1$.

When $B>0$, we have
\[
|\theta|=\left|\frac{\beta}{\alpha}\right|=\frac{\sqrt{A^2+4B}-|A|}{\sqrt{A^2+4B
}+|A|}=1-\frac{2|A|}{\sqrt{A^2+4B}+|A|}
\]
and with $|A|\geq 3$ we obtain
\[
|\theta|\geq 1+\frac{6}{\sqrt{9+4B}+3}
\]
and
\[
|\alpha|>\dfrac{\sqrt{9+4B}+3}2=:y.
\]
We need to show that
\[
\left|\dfrac{u_{8}}{u_{7}}\right|\geq y \frac{1-\left(1-\dfrac
3y\right)^{8}}{1+\left(1-\dfrac 3y\right)^{7}}>3.
\]
The inequality holds unless $y\leq 3$. However, $y\leq 3$ is impossible, since
it would imply $B\leq 0$ which we excluded.

Similar arguments for the cases $|A|=2$ and $|A|=1$ yield that
$\left|\dfrac{u_{8}}{u_{7}}\right|>3$ provided that $B>9$ and $B>3$,
respectively.
\end{proof}

A similar argument for Lucas sequences of second kind reveals the following.

\begin{lemma}
\label{lem:growth_comp}
 Let $v=(v_n)_{n=0}^\infty$ be the Lucas sequence of second kind corresponding to the pair
$(A,B)$, with $A^2+4B>0$. Assume that $n>n'$. Then $|u_n/u_{n'}|>3$ if $n$ is
even or $n\geq 7$, unless
\begin{itemize}
 \item $B<0$ and $|A|\leq 7$ or
 \item $|A|=1$ and $0<B\leq 14$ or
 \item $|A|=2$ and $0<B\leq 3$.
\end{itemize}
\end{lemma}

\begin{proof}
 Since the proof of this Lemma is similar to that of Lemma
\ref{lem:growth_Lucas}, we omit most of the details. Note that when $|A|=1$ we
have $|v_1/v_0|=1/2<1$ and we can only deduce $|v_n/v_{n-1}|\geq 1$ provided
that $n\geq 2$. Thus we shall ensure $|v_7|\geq 6$. Direct computation yields
that indeed
\[
|v_7|=7B^3+14B^2+7B+1
\]
when $|A|=1$. Since $|A|=1$ implies that $B\geq 1$, we get $|v_7|>6$.
\end{proof}

Now we apply Lemma \ref{lemmatech} in combination with Lemma
\ref{lem:growth_Lucas} and find that if $u$ admits a three term
arithmetic progression, then one of the equations in \eqref{equation:klm} must
have a solution with even $m\leq 6$. Writing the terms of $u$ as polynomials in
$A$ and $B$ we have to deal with finitely many equations. We only present the main
ideas through one example for each different type of possible equations, we treat the rest in a similar manner.

\textbf{Case $m=2$.} Since $u_0=0,\ u_1=1$ and $u_2=A$, we only get trivial
equations like $A-2=0$.

\textbf{Case $m=4$.} The corresponding equations are linear in $B$. For
example, the triple $(k,l,m)=(1,2,4)$ substituted into the second equation of
\eqref{equation:klm} gives us
\[
A^3+2AB+A-2=0
\]
implying
\[
B=\dfrac{-A^3-A+2}{2A}.
\]
Hence $A\mid 2$, i.e., $A=\pm1,\pm2$. However,
we see that one of the conditions $AB\neq 0$ and $A^2+4B>0$ fails in all cases. One can handle the other equations similarly. 

\textbf{Case $m=6$.} The corresponding equations are quadratic in $B$. Our
general strategy is to push the discriminant of each equation between two squares and solve the
parametric equation for each square between these lower and upper bounds. For
example, the triple $(k,l,m)=(0,3,6)$ and the third equation of
\eqref{equation:klm} gives
\[
-6AB^2+(-8A^3+1)B-2A^5+A^2=0.
\]
The discriminant is $D=16A^6+8A^3+1=(4A^3+1)^2$. However, solving this for $B$
we get the pair $(A,B)=(A,-A^2)$ and $A^2+4B>0$ fails to hold.

Let us consider another example. For instance, take the triple
$(k,l,m)=(0,2,6)$. Then we obtain $D=A^6+6A^2-3A$. But
$$A^6<A^6+6A^2-3A<(|A|^3+1)^2$$ provided that $|A|>3$ and we deduce that $D$
is never a square if $|A|>3$.
For those with $|A|\leq 3$ and $A\neq 0$ we obtain that $D$ is a square if
and only if $A=1$. In this case, $B=0$ which contradicts our assumption.

The same approach for each triple gives the complete list of three term arithmetic
progressions.

We are left with the finite number of sequences satisfying
\begin{itemize}
 \item $B<0$ and $|A|\leq 6$ or
 \item $|A|=1$ and $0<B\leq 9$ or
 \item $|A|=2$ and $0<B\leq 3$.
\end{itemize}
In these finitely many cases, we can use a growth argument to find
all three term arithmetic progressions.

Let us demonstrate this by an example, say $A=2$ and $B=1$ and $(u_n)_{n=0}^\infty$ being a Lucas sequence of first kind. One may handle all the other finitely many equations by a similar reasoning. First note
that in this specific case we have that $\alpha=1+\sqrt{2}$ and
$\beta=1-\sqrt{2}$. Assume that $u_k<u_l<u_m$ is a three term arithmetic
progression and consider the first equation of \eqref{equation:klm} which is
equivalent to
\[
\alpha^k+\alpha^m-2\alpha^l=\beta^k+\beta^m-2\beta^m.
\]
Let us assume for the moment that $m>l+1>k+1$. Then
\[
|\alpha|^m-2|\alpha|^{m-2}-|\alpha|^{m-3}
<|\alpha^k+\alpha^m-2\alpha^l|=|\beta^k+\beta^m-2\beta^m|<4|\beta|^k.
\]
Plugging in the concrete values of $\alpha$ and $\beta$ we immediately get
\[
(1+\sqrt{2})^{m-3}\left((1+\sqrt{2})^3-2(1+\sqrt{2})-1\right)=(4+3\sqrt
2)(1+\sqrt{2})^{m-3}<4
\]
which yields a contradiction unless $m<3$. Now it is easy to find all three term
arithmetic progressions in this case. Note that in case of $l>m$ or $k>m$ we may draw similar conclusions.

Hence we can assume that $m=l+1>k+1$. Then
\[
|\alpha|^m-2|\alpha|^{m-1}<|\alpha|^{k}+4|\beta|^k
\]
and inserting the concrete values of $\alpha$ and $\beta$ we get
\[
(1+\sqrt 2)^{m-2}<(1+\sqrt 2)^{k}+4.
\]
Assuming $k\leq m-3$ yields
\[
\sqrt{2}(1+\sqrt 2)^{m-3}<4,
\]
whence $m\leq 4$ and we easily find all three term arithmetic progressions
listed in Table \ref{table1} for this case. Finally, we have to consider the
case when $m=l+1=k+2$. Here we get the inequality
\[
\left|\alpha^m-2\alpha^{m-1}-\alpha^{m-2}\right|<4|\beta|^k
\]
which yields
\[
2(1+\sqrt 2)^{m-2}<4,
\]
that is, $m\leq 3$ and we find no additional three term arithmetic progressions.
Similar arguments for different orderings of $u_k,u_l,u_m$ reveal no further
solutions. Thus the case $A=2$ and $B=1$ is completely solved.

\section{The case $A^2+4B<0$}\label{Sec:imag}

Unfortunately, we found no effective method to resolve this case completely. The
reason has its roots in the use of the theory of $S$-unit equations. Indeed, the
indices $k,l,m$ corresponding to a non-trivial arithmetic progression
$u_k<u_l<u_m$ contained in a Lucas sequence $(u_n)_{n= 0}^\infty$ of first kind or a
Lucas sequence $(v_n)_{n= 0}^\infty$ of second kind satisfies the Diophantine equation
\begin{equation}\label{eq:AP}
\alpha^k+\alpha^m-2\alpha^l=\pm \left(\beta^k+\beta^m-2\beta^l\right),
\end{equation}
where the $\pm$ sign depends on whether we consider Lucas sequences of first or
second kind. In view of Theorem \ref{theorem1}, it is
crucial to discuss the number of solutions to $S$-unit equation
\eqref{eq:AP}.
In the real case (cf. Section \ref{Sec:real}), we resolved this $S$-unit
equation by
using the fact that $|\alpha|>|\beta|$ and that $\left|\dfrac{u_n}{u_{n-1}}\right|\simeq |\alpha|$ as $n\rightarrow \infty$.
If $A^2+4B<0$, then this is no longer true and we have to apply
the deep theory of $S$-unit equations.

In order to determine an upper bound for the number of solutions we prove a
series
of Lemmas which culminate in a proof for Theorem \ref{theorem1}.

Before we start with the first lemma, let us note that both $\alpha$ and $\beta$
are algebraic integers. Moreover, due to Theorem \ref{theorem2}, which was proved in
the previous section, we may assume that $\alpha$ and $\beta$ are conjugate
imaginary quadratic integers. Hence we can suppose that $\alpha$ and
$\beta$ are not units, otherwise $\alpha/\beta$ would be a root of unity. Thus
we may suppose that $|B|\geq 2$. Write $K=\Q(\alpha)=\Q(\beta)$ and denote by
$\sigma$ the unique, nontrivial $\Q$-automorphism of $K$, i.e.,
$\mathrm{Gal}(K/\Q)=\{\mathrm{id},\sigma\}$. Since $K$ is an imaginary quadratic
field, we also have that $\sigma(\epsilon)=\epsilon^{-1}$, where $\epsilon$ is
any unit in $K$.

For fixed $\alpha$ and $\beta$ we denote by $C_1$ and $C_2$ the number of
non-trivial three term arithmetic progressions of $(u_n)_{n=0}^{\infty}$ and
$(v_n)_{n=0}^{\infty}$, respectively and write $C=\max\{C_1,C_2\}$. Actually,
our upper bounds for $C_1$ and $C_2$ coincide in each instance and hence we use
only the quantity $C$.

Let $\Gamma=\langle\alpha,\beta\rangle\leq \bar\Q^*$ be the
multiplicative group generated by $\alpha$ and $\beta$. We denote by
$A=A(a_0,\dots,a_r)$ the number of non-degenerate, projective solutions
$\lbrack x_0:\dots:x_r \rbrack\in\mathbb{P}^r(\Gamma)$ of the weighted, homogeneous $S$-unit
equation
\begin{equation}\label{eq:S-unit}
 a_0u_0+\dots+a_ru_r=0.
\end{equation}
with $a_0,\dots,a_r\in \mathbb{C}$, in unknowns $u_0,\dots,u_r\in\Gamma$. By a non-degenerate solution we mean a solution to equation \eqref{eq:S-unit} such that no subsum on the left hand side
of equation \eqref{eq:S-unit} vanishes. Upper bounds $A(a_0,\dots,a_r)\leq
A(r,s)$, where $s$ denotes the rank of $\Gamma$, have been found most prominently
by Evertse, Schlickewei and Schmidt \cite{Evertse:2002} and Amoroso and Viada
\cite{Amoroso:2009} in the most general case. In particular, we use the bounds
due to Schlickewei and Schmidt \cite{Schlickewei:2000} which in our case yield better
results although they depend on $d=\lbrack K:\Q \rbrack$, as well.

As already noted above, three term arithmetic progressions in the Lucas
sequences $(u_n)_{n=0}^\infty$ and $(v_n)_{n=0}^\infty$ yield solutions to
\eqref{eq:AP}. Thus we study the weighted $S$-unit equation
\begin{equation}\label{eq:S-unit2}
x_1+x_2-2x_3=\pm(y_1+y_2-2y_3),
\end{equation}
where $x_1,x_2,x_3,y_1,y_2,y_3\in \Gamma=\langle\alpha,\beta\rangle$ and, in
particular, where $x_1=\alpha^k,x_2=\alpha^m,x_3=\alpha^l,
y_1=\beta^k,y_2=\beta^m,y_3=\beta^l$.

Before we start investigating $S$-unit equations of type \eqref{eq:S-unit2} we state a theorem due to Beukers~\cite[Theorem 2]{Beukers:1980}
and draw some simple conclusions from it.

\begin{lemma}[Beukers~\cite{Beukers:1980}]
Let the non-degenerate recurrence sequence of rational integers $(u_n)_{n=0}^\infty$ be given by $u_0\geq 0$, $\gcd(u_0, u_1)=1$,
$u_n = Au_{n-1} - Bu_{n-2}$, with $A, B\in\Z$ and $A\geq 0$. Assume that $A^2 - 4B < 0$. If $u_m =\pm u_0$ has
more than three solutions $m$, then one of the following cases holds:
\begin{align*}
A=1, B=2,& u_0=u_1=1& \text{when\ \ }& \ m=0,1,2,4,12;\\
A=1, B=2,& u_0=1, u_1=-1& \text{when\ \ }& \ m=0,1,3,11;\\
A=3, B=4,& u_0=u_1=1& \text{when\ \ }& \ m=0,1,2,6;\\
A=2, B=3,& u_0=u_1=1& \text{when\ \ }& \ m=0,1,2,5.
\end{align*}
\end{lemma}

From this Lemma we can easily conclude the following statement.

\begin{lemma}\label{lem:multiplicity}
Let $(u_n)_{n=0}^\infty$ be a Lucas sequence of first or second kind. Then for a fixed integer $\lambda$ the equation
$\lambda=u_m$ has at most $3$ solutions $m$.
\end{lemma}

\begin{proof}
Let $m_0$ be the smallest solution to $\lambda=u_m$. Consider instead of the original sequence the sequence $a_k=u_{m_0+k}/g(-1)^{k\delta}$,
where $g=\sign(u_{m_0})\gcd(u_{m_0},u_{m_0+1})$ and $\delta=\frac{1-\sign(A)}2$ (see also the remarks following Theorem 2 in \cite{Beukers:1980}).
Therefore we conclude that an equation of the form  $\lambda=|u_m|$ has more than three solutions if for the smallest solution $m_0$, also $m_0+1$ is a solution and $(A,B)=(\pm 1,-2),(\pm 3,-4)$ or $(\pm 2,-3)$.

However, the equation $|u_{m_0}|=|u_{m_0+1}|$ is equivalent to the equation
$$\frac{\alpha \pm 1}{\beta \pm 1}=\left(\frac \beta \alpha\right)^{m_0}$$
in case that $(u_n)_{n=0}^\infty$ is a Lucas sequence of first kind and
$$\frac{\alpha \pm 1}{\beta \pm 1}=- \left(\frac \beta \alpha\right)^{m_0}$$
in case that $(u_n)_{n=0}^\infty$ is a Lucas sequence of second kind.
For all pairs of $(\alpha,\beta)$ we solve these equations and find for instance, that $m_0\leq 2$ (if a solution exists at all). By computing the
values of all possible further solutions we find that indeed no equation of the form $\lambda=u_m$ has more than three solutions.
\end{proof}

\begin{remark}
Beukers \cite[Corollary on page 267]{Beukers:1980} already showed this result in case of Lucas sequences of first kind. In fact, he proved more and 
listed all the cases when there can be three solutions. To the authors knowledge the case of Lucas sequences of second kind
has not been treated since then.
\end{remark}

Let us start by investigating the number of non-degenerate solutions to $S$-unit equation \eqref{eq:AP}:

\begin{lemma}\label{lem:nondeg}
There are at most $A(5,2)$ three term arithmetic progressions contained in a
Lucas sequence $(u_n)_{n=0}^\infty$ of first kind that yield non-degenerate solutions to \eqref{eq:AP}.

The same statement also holds for Lucas sequence $(v_n)_{n=0}^\infty$ of second kind.
\end{lemma}

\begin{proof}
Consider the weighted $S$-unit equation
\begin{equation}\label{eq:S-unit2-special}
x_1+x_2-2x_3=y_1+y_2-2y_3
\end{equation}
with unknowns $x_1,x_2,x_3,y_1,y_2,y_3\in\Gamma$. This has at most
$A(5,2)$ non-degenerate, projective solutions. Thus there exists a set
$\mathcal C$, with $|\mathcal C|\leq A(5,2)$, of sextuples
$(1,c_2,c_3,c_4,c_5,c_6)$ such that for any solution we have
$x_2/x_1=c_2$, $x_3/x_1=c_3$, $y_1/x_1=c_4$, $y_2/x_1=c_5$ and
$y_3/x_1=c_6$. Assume now that a solution comes from an arithmetic
progression $u_k<u_l<u_m$ in a Lucas sequence $(u_n)_{n=0}^\infty$
of first kind. Then we have
\begin{equation}\label{eq:only-one}
\frac{\alpha^k}{\beta^k}=\frac{1}{c_4}, \quad
\frac{\alpha^m}{\beta^m}=\frac{c_2}{c_5},\quad
\frac{\alpha^l}{\beta^l}=\frac{c_3}{c_6}
\ \ \ \text{or}\ \ \ \frac{\beta^k}{\alpha^k}=\frac{1}{c_4}, \quad
\frac{\beta^m}{\alpha^m}=\frac{c_2}{c_5},\quad
\frac{\beta^l}{\alpha^l}=\frac{c_3}{c_6}
\end{equation}
for some $(1,c_2,c_3,c_4,c_5,c_6)\in{\mathcal C}$. Since by assumption
$\alpha/\beta$ is not a root of unity, for every sextuple
$(1,c_2,c_3,c_4,c_5,c_6)\in\mathcal C$, there exists at most one triple
$(k,l,m)$ such that any set of identities of \eqref{eq:only-one} is
satisfied, i.e., $u_k<u_l<u_m$ is an arithmetic progression. Hence there
exist at most $A(5,2)$ three term arithmetic progressions in a Lucas
sequence of first kind $(u_n)_{n=0}^\infty$, yielding a
non-degenerate solution to \eqref{eq:AP}.

In the case of a Lucas sequence of second kind $(v_n)_{n=0}^\infty$,
we consider instead of \eqref{eq:S-unit2-special} the Diophantine
equation
\[
x_1+x_2-2x_3=-(y_1+y_2-2y_3)
\]
and derive the same conclusion as before.
\end{proof}

Before we start dealing with vanishing subsums in equation \eqref{eq:AP} we note
that $\alpha$ and $\beta$ are always multiplicatively independent. Otherwise, we
would have $\alpha^t=\beta^s$ for some integers $t,s$. Taking the conjugates we
get $\alpha^s=\beta^t$ and hence $(\alpha/\beta)^{t+s}=1$ which we did not allow
in the definition.

To deal with vanishing two-term subsums we use the following lemma.

\begin{lemma}\label{lem:mult-dep-2}
For fixed non-zero complex numbers $a,b$ and $c$ either there exists at most one solution
$x,y\in\Z$ to $a^x=cb^y$ or $a$ and $b$ are
multiplicatively dependent.
\end{lemma}

\begin{proof}
Assume that two distinct solutions $x,y$ and $x',y'$ exist. Then we have
\[
\frac{a^x}{b^y}=c=\frac{a^{x'}}{b^{y'}},
\]
hence
\[
a^{x-x'}=b^{y-y'}.
\]
Thus $a$ and $b$ are multiplicatively dependent since by
assumption at least one exponent does not vanish.
\end{proof}

In view of Lemma \ref{lem:nondeg}, we are left to deal with vanishing subsums,
which may occur in equation~\eqref{eq:AP}. Of course, no one-term vanishing
subsum exists. First, consider two-term vanishing subsums.

\begin{lemma}\label{lem:two-vanish}
There are at most $3A(3,2)+30$  three term arithmetic progressions in a Lucas
sequence $(u_n)_{n=0}^\infty$ of first kind or a Lucas sequence
$(v_n)_{n=0}^\infty$ of second kind that yield a solution to \eqref{eq:AP} such
that a two-term subsum vanishes.
\end{lemma}

\begin{proof}
First, we consider the Lucas sequences $(u_n)_{n=0}^\infty$ of first kind. In this case, there
exist exactly $15$ possible vanishing two-term subsums in equation
\eqref{eq:AP}. We can divide these $15$ subsums into five classes, namely
\begin{description}
  \item[Case I] $x_1=-x_2$, $y_1=-y_2$,
  \item[Case II] $x_1=y_2$, $x_2=y_1$,
  \item[Case III] $x_1=2x_3$, $x_2=2x_3$, $y_1=2y_3$, $y_2=2y_3$,
  \item[Case IV] $x_1=-2y_3$, $x_2=-2y_3$, $y_1=-2x_3$, $y_2=-2x_3$.
  \item[Case V] $x_1=y_1$, $x_2=y_2$, $x_3=y_3$,
\end{description}

For each one we pick an equation and discuss it in detail. Since one can treat the other
equations in the same class by exactly the same arguments, we do
not give details for them.

\noindent\textbf{Case I:} The equation $x_1=-x_2$ implies that
$\alpha^k=-\alpha^m$ and we get $\alpha^{k-m}=-1$. Thus $\alpha$ and hence
$\beta$ are roots of unity, which is excluded.

\noindent\textbf{Case II:} The equation $x_1=y_2$ implies that
$\alpha^k=\beta^m$, i.e., $\alpha$ and $\beta$ are multiplicatively dependent,
which cannot be the case.

\noindent\textbf{Case III:} The equation $x_1=2x_3$ implies that
$\alpha^k=2\alpha^l$, i.e., $\alpha^{k-l}=2$. Thus $\alpha=\pm \sqrt 2$ or
$\alpha=2$, both a contradiction to our assumption that $\alpha$ is imaginary
quadratic.

\noindent\textbf{Case IV:} If $x_1=-2y_3$ we get the equation
$\alpha^k=-2\beta^l$ and an by an application of Lemma \ref{lem:mult-dep-2}
there exists at most one pair $(k,l)$ satisfying $\alpha^k=-2\beta^l$ or
$\alpha$ and $\beta$ are multiplicatively dependent. The latter is not
possible. However, if the pair $(k,l)$ is fixed then $u_l$ and $u_k$ and therefore
also $u_m$ is fixed. Due to Lemma \ref{lem:multiplicity} there are at most three possiblities for $m$,
i.e., there are at most three ways to extend the pair $(k,l)$ to a triple $(k,l,m)$ such that
$u_k<u_l<u_m$ is an arithmetic progression.

\noindent\textbf{Case V:} If $x_1=y_1$, then we have $\alpha^k=\beta^k$ and
therefore $(\alpha/\beta)^k=1$. Hence $\alpha/\beta$ is a
root of unity unless $k=0$.

So far we have proved that one of the following statements holds:
\begin{itemize}
\item we obtain at most $12$ additional solutions coming from Case IV,
\item $klm=0$.
\end{itemize}

Concerning the last case we assume that $k=0$ first. We get the $S$-unit
equation
\begin{equation}\label{eq:AP-k=0}
x_2-2x_3=y_2-2y_3.
\end{equation}
By the same arguments as in the proof of Lemma \ref{lem:nondeg} we obtain that
there are at most $A(3,2)$ three term arithmetic progressions which come from
non-degenerate solutions of \eqref{eq:AP-k=0}. Further, we have two vanishing
two-term subsums falling into Case III, two falling into Case IV and two
into Case V. Case III yields no additional solutions. Case IV yields at most $3$
additional solution, hence in total at most six additional solutions. Finally,
note that in Case V we can now exclude that $l$ or $m$ vanishes, otherwise we
would obtain $k=l$ or $k=m$ and we get no additional solution. Thus $k=0$ yields
at most $A(3,2)+6$ additional solutions.

By the same arguments it is possible to treat the cases of $l=0$ and $m=0$. We
omit the details. However, the case that $klm=0$ yields at most $3A(3,2)+18$
additional solutions. Therefore we have at most $3A(3,2)+30$ solutions.
\end{proof}

Let us note that the proof of Lemma \ref{lem:two-vanish}, in particular its last
part, shows the following.

\begin{lemma}\label{lem:kmn=0}
There are at most $A(3,2)+6$ three term arithmetic progressions $u_k<u_l<u_m$ in
a Lucas sequence of first or second kind such that $k=0$. The same statements hold
if we replace $k=0$ by $m=0$ or $l=0$, respectively.
\end{lemma}

Since a vanishing four- or five-term subsum implies a two- or one-term vanishing
subsum, respectively, we are left to consider vanishing three term subsums.

\begin{lemma}\label{lem:three-vanish}
At least one of the following statements holds:
\begin{itemize}
\item There are at most $18A(2,2)+9$ three term arithmetic progressions in a
Lucas sequence of first or second kind that yield a solution to \eqref{eq:AP}
such that a three term subsum vanishes,
\item or $\alpha$ and $\beta$ are quadratic irrational numbers and if
$u_k<u_l<u_m$ is an arithmetic progression, then the companion polynomial of the
Lucas sequence $(u_n)_{n=0}^\infty$ of first kind divides $X^k-2X^l+X^m$,
\item or $\alpha$ and $\beta$ are quadratic irrational numbers and if
$v_k<v_l<v_m$ is an arithmetic progression, then the companion polynomial of the
Lucas sequence $(v_n)_{n=0}^\infty$ of second kind divides $X^k-2X^l+X^m$.
\end{itemize}
\end{lemma}

\begin{proof}
Let us consider the case of Lucas sequences of first kind. Lucas sequences of second kind can
be treated by the same means.

There are $20$ possible, different vanishing three term subsums of
\eqref{eq:AP}. These appear in pairs and we can differentiate three different
types:
 \begin{align*}
  I  &:& x_1+x_2-2x_3&=0,& 0&=y_1+y_2-2y_3,\\
  II &:& x_1+x_2&=y_1,& -2x_3&=y_2-2y_3,\\
  II &:& x_1+x_2&=y_2,& -2x_3&=y_1-2y_3,\\
  III&:& x_1+x_2&=-2y_3,& -2x_3&=y_1+y_2,\\
  II &:& x_1-2x_3&=y_1,& x_2&=y_2-2y_3,\\
  III&:& x_1-2x_3&=y_2,& x_2&=y_1-2y_3,\\
  II &:& x_1-2x_3&=-2y_3,& x_2&=y_1+y_2,\\
  II &:& x_1&=y_1+y_2,& x_2-2x_3&=-2y_3,\\
  II &:& x_1&=y_1-2y_3,& x_2-2x_3&=y_2,\\
  III&:& x_1&=y_2-2y_3,& x_2-2x_3&=y_1.
 \end{align*}

Let us take the equations of type II first. Pick the first pair of equations in
the list above (the other cases run completely analogously). Consider the left
hand side equation
\begin{equation}\label{eq:CaseII}
x_1+x_2=y_1.
\end{equation}
By the theory of $S$-unit equations there are at most $A(2,2)$ pairs $(c_1,c_2)$
such that $x_1/y_1=(\alpha/\beta)^k=c_1$ and $x_2/y_1=\alpha^m/\beta^k=c_2$.
These two relations determine the pair $(k,m)$ uniquely. Hence there are at most
$A(2,2)$ pairs $(k,m)$ that fulfill the equation. With the pair $(k,m)$ fixed, the value of $u_l$ is
fixed as well and by Lemma \ref{lem:multiplicity} there are at most three ways to extend the pair
$(k,l)$ to a triple $(k,l,m)$ such that $u_k<u_l<u_m$ is an arithmetic progression.
We deduce that there exist at most $3A(2,2)$ triples $(k,l,m)$ satisfying a pair of equations of type II.

Consider now equations of type I. Rewriting the two equations we get
$\alpha^k+\alpha^m-2\alpha^l=0=\beta^k+\beta^m-2\beta^l$. Thus in this case, the
minimal
polynomial of $\alpha$ and $\beta$, which is also the companion polynomial of
the sequence, divides $X^k-2X^l+X^m$.

Finally, we turn to equations of type III and consider the first equation of
this type in detail. We can handle the others similarly. Assume for the moment
that there exists a prime ideal $\mathfrak p$ with $\mathfrak p\nmid (2)$ and
$v_{\mathfrak p}(\alpha)>0$. Let $\mathfrak p_2$ be a prime ideal lying above
$(2)$ and suppose that $m>k$. We consider the equation $x_1+x_2=-2y_3$ from a
$\mathfrak p$-adic point of view. Write $v_{\mathfrak p}(\alpha)=\xi_{\mathfrak p}$ and
$v_{\mathfrak p}(\beta)=\eta_{\mathfrak p}$. Since the two smallest $\mathfrak
p$-adic valuations of the summands have to coincide, we have $v_{\mathfrak
p}(x_2)=k\xi_{\mathfrak p}=l\eta_{\mathfrak p}=v_{\mathfrak p}(2y_3)$.
Considering $\mathfrak p_2$-adic valuations we obtain $k\xi_{\mathfrak
p_2}=l\eta_{\mathfrak p_2}+\delta$, where $\delta=1,2$ depending on whether
$(2)$ is ramified or not. We get the following system of linear equations
\begin{align*}
 k\xi_{\mathfrak p_2}-l\eta_{\mathfrak p_2}&=\delta\\
 k\xi_{\mathfrak p}-l\eta_{\mathfrak p}&=0
\end{align*}
which has either no solution or at most one solution. Note that by assumption $\xi_{\mathfrak p}>0$.
Thus in this case, equations of type III yield at most one pair $(k,l)$ that extends to a triple $(k,l,m)$
such that $u_k<u_l<u_m$ is an arithmetic progressions. Thus by Lemma \ref{lem:multiplicity} we have at most
three additional three term arithmetic
progression.

If $v_\mathfrak{p}(\alpha)=0$, but
$v_\mathfrak{p}(\beta)>0$, then we take the equation $-2x_3=y_1+y_2$ instead and
draw the same conclusions.

In view of the paragraph above, we may assume that the only prime ideals
dividing $(\alpha)$ and $(\beta)$ are lying above $(2)$. We distinguish now
whether $(2)$ splits, ramifies or is inert above $K$.

We start with the case when $(2)$ is inert, i.e., $(2)$ is a prime in $K$. Then
$(\alpha) = (2)^x$
and $(\beta) = (2)^y$ with some non-negative integers $x,y$. Since
$\alpha$ and
$\beta$ are conjugate, we get $x=y$. Hence $\alpha/\beta$ is a unit.
However,
recalling that $\alpha,\beta$ are from a quadratic imaginary field, this
yields that
$\alpha/\beta$ is a root of unity, which is a contradiction.

One may prove the case where $(2)$ ramifies by similar means as the
inert case. In this
case,we have $(2)=\mathfrak{p}^2$ with some prime ideal $\mathfrak{p}$.
Note that $\mathfrak p$ need not be principal although $\mathfrak p^2$ is principal.
Thus $(\alpha)=\mathfrak{p}^x$
and
$(\beta) = \mathfrak{p}^y$  with some non-negative integers $x,y$. Taking norm, we
get $x=y$,
and similarly as above we obtain that $\alpha/\beta$ is a root of unity,
which is impossible.

Finally, let us consider the case when $(2)$ splits, that is, we can write
$(2)=\mathfrak p_1\mathfrak p_2$. Assume that $(\alpha)=\mathfrak
p_1^{\xi}\mathfrak p_2^{\eta}$, then $(\beta)=\mathfrak p_1^{\eta}\mathfrak
p_2^{\xi}$. Considering the equation $x_1+x_2=-2y_3$ from a $\mathfrak p_1$-adic
and a $\mathfrak p_2$-adic viewpoint, respectively we get the following system
of linear equations under the assumption that $m>k$:
\begin{equation}\label{eq:valuations}
\begin{split}
v_{\mathfrak p_1}(x_2)=k\xi &= l\eta +1= v_{\mathfrak p_1}(2y_3),\\
v_{\mathfrak p_2}(x_2)=k\eta &= l\xi +1= v_{\mathfrak p_2}(2y_3).
\end{split}
\end{equation}
Note that we obtained this system due to the fact that the smallest
$\mathfrak{p}$-adic valuations must coincide. Solving linear system
\eqref{eq:valuations} yields $k(\xi^2-\eta^2)=(\xi-\eta)$. Therefore either
$\xi=\eta$ or $k(\xi+\eta)=1$. The first case yields $\alpha=2^\xi\epsilon$ and
$\beta=2^\xi\epsilon^{-1}$ for some unit $\epsilon$ which yields that
$\alpha/\beta=\epsilon^2$ is a root of unity, a contradiction. When
$k(\xi+\eta)=1$ we have $k=1$ and $(\xi,\eta)=(0,1),(1,0)$. In any case, we
obtain from the linear system \eqref{eq:valuations} that $l=-1$, a contradiction
to our assumption that $k,l,m$ are nonnegative integers.
\end{proof}

So we are left with the problem which quadratic polynomials divide
$X^k-2X^l+X^m$. In particular, we may assume that the companion polynomial is
irreducible, otherwise $\alpha$ and $\beta$ would be rational integers which
case is covered by Theorem \ref{theorem2}. We need to find all quadratic factors
of polynomials of the type:
\begin{itemize}
\item $X^a-2X^b+1$ or
\item $X^a+X^b-2$ or
\item $2X^a-X^b-1$.
\end{itemize}

Since all roots of the first polynomial are units, and all the roots of the
third polynomial are either non-integral algebraic numbers or units, it is enough to
study only polynomials of the form $X^a+X^b-2$.

We state the following result due to Pint\'{e}r and Ziegler \cite[Lemma
3]{pinterziegler}.

\begin{lemma}\label{lem:Pinter}
 Let $a>b>0\in\mathbb{Z}$ and let $f(X)=X^a+X^b-2$ be a
polynomial. Then $f(X)=h(X)g(X)$ factors into the
monic polynomials $h(X)$ and $g(X)$ over $\mathbb{Z}$, where the only roots of
$h(X)$ are roots of unity and $g(X)$ is irreducible. Moreover
$h(X)|X^{\gcd(a,b)}\pm 1$.
\end{lemma}

We are interested in finding the quadratic factors of $X^{a}+X^{b}-2$. By Lemma
\ref{lem:Pinter} we only have to consider the cases $(a,b)=(3,2),(3,1),(2,1)$ or
$(4,2)$. Thus the companion polynomial is either $X^2+2X+2$ or $X^2+X+2$ or
$X^2+2$.

Since $X^2+X+2$ is the only polynomial with negative discriminant which is also a companion polynomial of a
non-degenerate binary recurrence, we obtain that either $(A,B)=(-1,-2)$ or there
are at most $A(5,2)$ three term arithmetic progressions coming from
non-degenerate solutions to \eqref{eq:AP} (cf. Lemma \ref{lem:nondeg}),
$3A(3,2)+30$ three term arithmetic progressions coming form two-term vanishing
subsums of \eqref{eq:AP} (cf. Lemma \ref{lem:two-vanish}) and $18A(2,2)+9$
three term arithmetic progressions coming form three term vanishing subsums of
\eqref{eq:AP} (cf. Lemma \ref{lem:three-vanish}). We summarize these results.

\begin{proposition}
\label{prop:exact}
Let $u=(u_n)_{n=0}^\infty$ and $v=(v_n)_{n=0}^\infty$ be Lucas sequences of
first and second kind corresponding to the pair $(A,B)$ with $A^2+4B<0$. Then
$u$ and $v$ admit at most
\[
C=A(5,2)+3A(3,2)+18A(2,2)+39
\]
three term arithmetic progressions unless $u$ or $v$ correspond to
$(A,B)=(-1,-2)$ respectively. In case of $(A,B)=(-1,-2)$ the recurrences $u$
and $v$ admit infinitely many three term arithmetic progressions.
\end{proposition}

\begin{remark}
Since $X^2+X+2|X^3+X-2$, we get that $u_k<u_l<u_m$ and $v_k<v_l<v_m$ with
$(k,l,m)=(t+1,t,t+3)$ yield infinite families of three term arithmetic
progressions in $u$ and $v$, respectively.
\end{remark}

Theorem \ref{theorem1} is now an easy Corollary of Proposition
\ref{prop:exact}. Indeed, due to a result of Schlickewei and Schmidt
\cite{Schlickewei:2000}, we know that $A(k,s)\leq 2^{35B^3}d^{6B^2}$, where
$B=\max\{k+1,s\}$ and $d=\lbrack K:\Q \rbrack$, hence
\[
C\leq A(5,2)+3A(3,2)+18A(2,2)+39 <6.45\cdot 10^{2340}.
\]

\section{Acknowledgment}

The authors are grateful to the referee for her/his useful comments and suggestions.

\def\cprime{$'$}

\bigskip
\hrule
\bigskip

\noindent 2010 {\it Mathematics Subject Classification}: Primary 11B39. Secondary 11B25, 11D61.

\noindent \emph{Keywords: } Lucas sequences, arithmetic progressions, $S$-unit equations. 

\bigskip
\hrule
\bigskip

\noindent
(Concerned with sequences
\seqnum{A000045},
\seqnum{A039834},
\seqnum{A001045},
\seqnum{A000032}
and
\seqnum{A014551}.)

\bigskip
\hrule
\bigskip

\end{document}